\documentclass[12pt,twoside]{amsart}
\usepackage{mathrsfs,amsthm,amscd,amsmath,amssymb}
\usepackage[colorlinks,linkcolor=green,citecolor=blue, pdfstartview=FitH]{hyperref}
\usepackage{amscd}
\usepackage{actuarialsymbol}
\usepackage{amsfonts}
\usepackage{graphicx}
\usepackage{caption, subcaption}

\usepackage{booktabs}
 \usepackage{multirow}

\title
[A Perron-Frobenius comonotonic approximation for sums of lognormals]
{A Perron-Frobenius comonotonic approximation for sums of lognormals}
\author{Chunle Huang}
\address{Chunle Huang, School of Mathematics, Hunan University, Changsha, 410082, China.}
\email{2019044@hnu.edu.cn; 402961544@qq.com}

\newtheorem{thm}{Theorem}[section]
\newtheorem{lem}[thm]{Lemma}

\theoremstyle{definition}

\makeatletter
\let\uppercasenonmath\@gobble
\makeatother
\begin{document}
\bibliographystyle{amsalpha+}
\maketitle
\begin{abstract} 
In this note, we introduce a new comonotonic approximation for sums of lognormal random variables based on the famous Perron-Frobenius theorem. Unlike the Taylor-based (TB) and maximal variance (MV) methods, which construct the conditioning variable $\Lambda$ 
with coefficients depending solely on marginal distributions, our approach captures the joint dependence structure by using the Perron-Frobenius eigenvector of the covariance matrix $\Sigma$, which is assumed to be positive. This choice ensures the comonotonicity of the resulting approximation, yields simple correlation formulae, exploits the spectral properties of $\Sigma$ and embodies the principal direction of the multivariate risk. Numerical results show that the PF approximation significantly outperforms the TB and MV methods in quantifying tail risks, especially for CTE and ESF under high volatility and high confidence levels, establishing it as a robust and accurate tool for complex multivariate models. 
\end{abstract}
\section{Introduction} 
In this note, we consider the problem of approximating the distribution and risk measures of a random variable $S = \sum^n_{i = 1}\alpha_ie^{Z_i}$, where $\alpha_i > 0$ are positive numbers and $(Z_1, Z_2, ..., Z_n)^{'}$ is a multivariate random vector. In practice, the $\alpha_i$ can be viewed as future deterministic payments or saving amounts, while the random variables $Z_i$ describe underlying financial risks. For example, if $Z_i$ denotes the stochastic log-return over the period $[i, n]$ for each $i = 1, 2, ..., n$, then $S$ can be interpreted as the accumulated value at time $n$ of future deterministic savings. On the other hand, if $-Z_i$ denotes the log-return over the period $[0, i]$, then $S$ is the stochastic present value of a series of future deterministic payments. Since Gaussian models are widely applied in finance for modeling asset returns, we assume throughout this article that $(Z_1, Z_2, ..., Z_n)^{'}$ follows a multivariate normal distribution. In this case, $S$ becomes a sum of dependent lognormal random variables, making it analytically intractable to derive exact expressions for its distribution function and risk measures. Therefore, it is of great value to consider its approximations.

A variety of approximation techniques have been proposed in the literature. One prominent approach is moment matching, which replaces the unknown distribution of $S$ with a known distribution sharing the same first two moments. Given the lognormal nature of the individual terms in $S$, a lognormal approximation is a natural candidate; alternatively, a reciprocal Gamma approximation can be employed since the present value of a constant continuous perpetuity with lognormal returns, a limiting case of $S$, is reciprocal Gamma distributed, as discussed in \cite{D1990} and \cite{M1997}. For further details on moment matching methods and their applications, see \cite{KPW2012}, \cite{MP1998}, \cite{MR2000} and \cite{VHD2005b}. While computationally attractive, moment matching does not explicitly account for the dependence structure among the lognormal components, which can be critical for accurate tail risk measurement.

Another technique, and the one directly relevant to this note, relies on the concept of comonotonicity. Following the idea of \cite{RS1995}, the authors in \cite{KDG2000} propose to approximate the distribution function of $S$ by that of $S^l = \text{E}[S\, |\, \Lambda]$ for an appropriately chosen conditioning random variable $\Lambda$. This approach reduces the multivariate randomness of $(Z_1, Z_2, ..., Z_n)^{'}$ into a univariate source $\Lambda$. Furthermore, with a suitable choice of $\Lambda$, the resulting approximation $S^l$ forms a comonotonic sum. Then, risk measures related to the distribution function of $S$ are approximated by the corresponding risk measures of $S^l$. The calculations become particularly efficient for distortion risk measures, which are additive for sums of comonotonic random variables. For an extensive overview on the concept of comonotonicity, we refer to \cite{DDGK2006} and \cite{DDG2002a}; for its applications in actuarial science and finance, see \cite{DD2007}, \cite{DHL2008}, \cite{DVG2006}, \cite{DWY2000} and \cite{VDG2003}.

Several choices for $\Lambda$ have been proposed in the literature. One is based on Taylor's formula, which constructs $\Lambda$ as a linear transformation of a first-order approximation of $S$, see \cite{DDG2002b, KDG2000}. Another is based on the idea of maximizing the variance of $S^l$, which makes the variance of $S^l$ a linear transformation of a first-order approximation of the variance of $S$, see \cite{VDG2005a, VHD2005b}. Both approaches are computationally simple and often perform well in practice. However, their coefficient vectors depend only on the marginal moments of the $Z_i$'s and hence do not fully exploit the joint dependence structure encoded in the covariance matrix $\Sigma$. 
This limitation becomes particularly pronounced under high volatility 
and in extreme tails. 

Following the seminal works of \cite{DDG2002a, DDG2002b, DVG2005c, KDG2000, VCD2008}, we address this gap in this note by proposing a new comonotonic approximation based on the famous Perron-Frobenius theorem for regular matrices. Specifically, we construct the conditioning random variable $\Lambda$ via the Perron-Frobenius eigenvector of the covariance matrix $\Sigma$, which is assumed to be positive. This choice exploits the spectral properties of $\Sigma$ and guarantees that \(\Lambda\) embodies the principal direction of the multivariate risk, thereby incorporating the joint distribution of all the components. Consequently, the Perron-Frobenius (PF) approximation is structurally more consistent and robust, especially in high-volatility regimes where the joint distribution plays a critical role in determining extreme tail behaviors. As illustrated in Section \ref{429212}, at \(\sigma = 0.25\) and \(p = 0.95\), the PF approximation consistently outperforms the TB and MV methods in the CTE and ESF risk measures, with relative errors of \(-0.54\%\) and \(-0.23\%\), respectively, compared to \(-1.38\%\) and \(-4.26\%\) for TB, and \(-0.66\%\) and \(-2.16\%\) for MV.

\section{The Perron-Frobenius approximation}\label{429153}
Let $(Z_1, Z_2, ..., Z_n)^{'} \sim MN(\mu, \Sigma)$ be a random vector that follows a multivariate normal distribution. For any positive numbers  $\alpha_1, \alpha_2, ..., \alpha_n$ we consider the sum $S = \sum^n_{i = 1}\alpha_ie^{Z_i}$ and its approximation defined by $S^l \overset{\mathrm{def}}{=} \text{E}[S\, |\, \Lambda]$ where $\Lambda$ is a conditioning random variable. 
Following the work of  \cite{DDG2002a}, \cite{DDG2002b}, \cite{DVG2005c}, \cite{KDG2000} and \cite{VCD2008} we consider the conditioning random variable $\Lambda$ as a linear combination of $Z_1, Z_2, ..., Z_n$: $\Lambda = \sum^n_{j = 1}\lambda_jZ_j$, 
where $\lambda_1, \lambda_2, ..., \lambda_n$ are constants. We denote the mean and the variance of $\Lambda$ by $\text{E}[\Lambda]$ and $\sigma_{\Lambda}^2$, respectively. Then, $S^l$ can be written as
\begin{equation}\label{219079}
S^l = \sum^n_{i = 1}\alpha_i \exp\Big(\text{E}[Z_i] + \frac{1}{2}(1 - r_i^2)\sigma_{Z_i}^2 + r_i\sigma_{Z_i}\Phi^{-1}(V)\Big)
\end{equation}
where $V = \Phi(\frac{\Lambda - \text{E}[\Lambda]}{\sigma_{\Lambda}})$ is a random variable uniformly distributed on $(0, 1)$ and $r_i$ is the correlation coefficient between the two random variables $Z_i$ and $\Lambda$, i.e., 
$$r_i = \frac{\text{Cov}[Z_i, \Lambda]}{\sigma_{Z_i}\sigma_{\Lambda}} = \frac{1}{\sigma_{Z_i}\sigma_{\Lambda}}\sum^n_{j = 1}\lambda_j \text{Cov}[Z_i, Z_j], \,\, i = 1, 2, ..., n.$$
Note that the expected values of the random variables $S$ and $S^l$ are equal,  
whereas their respective variances are given by 
$$\text{Var}[S] = \sum^n_{i = 1}\sum^n_{j = 1}\alpha_i\alpha_j\text{E}[e^{Z_i}]\text{E}[e^{Z_j}](e^{\text{Cov}[Z_i, Z_j]} - 1)$$ 
and 
\begin{equation} \label{325090}
\text{Var}[S^l] = \sum^n_{i = 1}\sum^n_{j = 1}\alpha_i\alpha_j\text{E}[e^{Z_i}]\text{E}[e^{Z_j}](e^{r_ir_j\sigma_{Z_i}\sigma_{Z_j}} - 1). 
\end{equation} 

The following result gives a general condition for the conditioning random variable $\Lambda$ to define a comonotonic approximation $S^l$ to the original sum $S$.  
\begin{lem} \label{322088}
Let $(Z_1, Z_2, ..., Z_n)^{'} \sim MN(\mu, \Sigma)$. For any positive numbers  $\alpha_1, \alpha_2, ..., \alpha_n$ consider the sum $S = \sum^n_{i = 1}\alpha_ie^{Z_i}$ and its approximation defined by $S^l \overset{\mathrm{def}}{=} E[S\, |\, \Lambda]$ with $\Lambda = \sum^n_{j = 1}\lambda_jZ_j$. Then for any non-zero $x = (x_1, x_2, ..., x_n)^{'} \in \mathbb{R}_{+}^n$, $\lambda = \Sigma^{-1}x$ defines a comonotonic sum, that is, the components of $S^l$ are comonotonic in the sense that all of them are non-decreasing functions of $\Lambda$. 
\end{lem}
\begin{proof} 
From equation (\ref{219079}), it suffices to show that the inequality $r_i\sigma_{Z_i} \geq 0$ holds true for all $i = 1, 2, ..., n$, which is, however, very easy to check. Indeed, for any non-zero vector $x = (x_1, x_2, ..., x_n)^{'} \in \mathbb{R}_{+}^n$ we compute 
\begin{align*}  
r_i = \frac{\text{Cov}[Z_i, \Lambda]}{\sigma_{Z_i}\sigma_{\Lambda}} = \frac{\sum^n_{j = 1}\lambda_j\text{Cov}[Z_i, Z_j]}{\sigma_{Z_i}\sqrt{\sum^n_{k = 1}\sum^n_{l = 1}\lambda_k\lambda_l\text{Cov}[Z_k, Z_l]}},\,\, i = 1, 2, ...n 
\end{align*} 
or, equivalently $r\sigma_{Z} = \frac{1}{\sqrt{\lambda^{'}\Sigma\lambda}}\Sigma\lambda$, where $r\sigma_{Z}= (r_1\sigma_{Z_1}, ..., r_n\sigma_{Z_n})^{'}$ and $\lambda = (\lambda_1, ..., \lambda_n)^{'}$. By making the substitution $\lambda = \Sigma^{-1}x$ we find that 
\begin{align*}
r\sigma_Z = \frac{1}{\sqrt{x^{'}\Sigma^{-1}x}}x = \frac{1}{\sqrt{u^{'}\Sigma^{-1}u}}u 
\end{align*} 
where $u = \frac{1}{||x||}x$ is the standardization of $x$. 
It follows that 
\begin{equation} \label{324107}
r_i\sigma_{Z_i} = \frac{1}{\sqrt{x^{'}\Sigma^{-1}x}}x_i = \frac{1}{\sqrt{u^{'}\Sigma^{-1}u}}u_i \geq 0, \,\, i = 1, 2, ..., n. 
\end{equation}
This means that $\lambda = \Sigma^{-1}x$ defines a comonotonic sum, as desired. 
\end{proof}

From the proof of Lemma \ref{322088}, in order to obtain comonotonic approximations for the sum $S$, it suffices to consider the case where $x = (x_1, x_2, ..., x_n)^{'} \in \mathbb{R}_{+}^n$ satisfies $||x|| = 1$. 
There are infinitely many choices of $x \in \mathbb{R}_+^n$ (with $\|x\|=1$) that yield comonotonic approximations. Among these, we seek one that is both accurate and computationally tractable. A natural criterion, standard in the comonotonic approximation literature \cite{VCD2008}, \cite{VDG2005a} and \cite{VHD2005b}, is to maximize $\mathrm{Var}[S^l]$: indeed, since $S^l \le_{cx} S$, a larger variance brings the approximation closer to $S$ in convex order, see \cite{DDG2002a} and \cite{KDG2000} for details. We therefore aim to choose $\Lambda$ so that $\mathrm{Var}[S^l]$ is as large as possible. The Perron-Frobenius theorem provides a natural way to achieve this, as we now explain after recalling some basic facts from matrix analysis.
A matrix or vector is called positive (resp. nonnegative) if all its entries are positive (resp. nonnegative). A nonnegative matrix $A \in \mathbb{R}^{n\times n}$ is called regular if $A^k$ is positive for some integer $k\ge1$; any positive matrix is regular. 
We will use the notation $x > y$ to mean that $x - y$ is positive and the notation $x \geq y$ to mean that $x - y$ is nonnegative. 
We recall the Perron-Frobenius theorem for regular matrices, see \cite{BR1997}, \cite{BP1994}, \cite{DH1953}, \cite{M1988} and \cite{PSC2005}.

\begin{lem} [Perron-Frobenius theorem for regular matrices]
Suppose that $A \in \mathbb{R}^{n \times n}$ is nonnegative and regular, i.e., $A^k > 0$ for some $k \geq 1$, then 
\begin{enumerate} 
\item there is an eigenvalue $\lambda_{pf}$ of $A$ that is real and positive, with positive eigenvectors; 
\item for any other eigenvalue $\lambda$ of $A$, we have $|\lambda| < \lambda_{pf}$; 
\item the eigenvalue $\lambda_{pf}$ is simple, i.e., has multiplicity one, and corresponds to a $1\times 1$ Jordan block.  
\end{enumerate} 
The eigenvalue $\lambda_{pf}$ is called the Perron-Frobenius eigenvalue of $A$. The associated positive eigenvectors are called the Perron-Frobenius eigenvectors and are unique, up to positive scaling. Denote the Perron-Frobenius eigenvector with unit norm by $x_{pf}$. 
\end{lem}

Hereafter, we assume that the covariance matrix $\Sigma$ is positive, which means that $Z_1, Z_2, ..., Z_n$ are positively correlated in the sense that $\text{Cov}[Z_i, Z_j] > 0, \,\, \text{for all}\,\, i, j = 1, 2, ...,n.$ 
Let $\lambda_{pf}$ be the Perron-Frobenius eigenvalue of $\Sigma$, and let $x_{pf} = (\hat{x}_1,\dots,\hat{x}_n)'$ be the corresponding unit-norm eigenvector, so that $\hat{x}_i > 0$ for all $i$ and $\|x_{pf}\|=1$. Combining this with Lemma \ref{322088} yields the following result.

\begin{thm} \label{322134} 
Let $(Z_1, ..., Z_n)^{'} \sim MN(\mu, \Sigma)$ with $\Sigma > 0$. For any positive numbers $\alpha_1, \alpha_2, ..., \alpha_n$ consider the sum $S = \sum^n_{i = 1}\alpha_ie^{Z_i}$ and its approximation $S^l \overset{\mathrm{def}}{=} E[S\, |\, \Lambda]$ with $\Lambda = \sum^n_{j = 1}\lambda_jZ_j$. Let $\lambda_{pf}$ and $x_{pf} = (\hat{x}_1, ..., \hat{x}_n)^{'}$ be the Perron-Frobenius eigenvalue and eigenvector with norm 1 of $\Sigma$, respectively. Then $\lambda = \Sigma^{-1}x_{pf}$ defines a comonotonic sum. Moreover, the correlation between $Z_i$ and $\Lambda$ is given by $$r_i = \sqrt{\lambda_{pf}}{\hat{x}_i}/{\sigma_{Z_i}}, \,\, \text{for all} \,\, i = 1, 2, ..., n.$$
\end{thm}
\begin{proof} 
Since $\Sigma$ is positive, the Perron-Frobenius theorem ensures 
$\hat{x}_i > 0$ for all $i$. From Lemma \ref{322088}, $\lambda = \Sigma^{-1}x_{pf}$ 
defines a comonotonic sum. For the correlation between $Z_i$ and the 
conditioning random variable $\Lambda$, using (\ref{324107}) gives
\begin{equation*} 
r_i = \frac{1}{\sqrt{x_{pf}^{'}\Sigma^{-1}x_{pf}}}\hat{x}_i / \sigma_{Z_i} = \frac{1}{\sqrt{\lambda_{pf}^{-1}||x_{pf}||^2}}\hat{x}_i/\sigma_{Z_i} = \sqrt{\lambda_{pf}}\hat{x}_i/\sigma_{Z_i}, \,\, i = 1, 2, ..., n
\end{equation*} 
since $||x_{pf}|| = 1$. This completes the proof. 
\end{proof}

Theorem \ref{322134} also extends to the case $\mathrm{Cov}[Z_i,Z_j] \ge 0$; see \cite{BR1997, BP1994, DH1953, M1988} for details. This follows from the extension of the Perron-Frobenius theorem to nonnegative regular matrices. We refer to the comonotonic approximation constructed in Theorem \ref{322134} as the Perron-Frobenius (PF) approximation.
The PF approximation has several appealing features. First, it guarantees comonotonicity whenever $\Sigma$ is positive. Second, it is 
computationally straightforward, requiring only the inverse and the 
Perron-Frobenius eigenvector of $\Sigma$. Third, the correlations 
$r_i$ are explicitly available from Theorem \ref{322134}. Finally, from (\ref{325090}) and (\ref{324107}), 
for any $x \in \mathbb{R}_+^n$ with $\|x\|=1$,
\[
\mathrm{Var}[S^l] = \sum_{i,j=1}^n \alpha_i\alpha_j \mathrm{E}[e^{Z_i}]
\mathrm{E}[e^{Z_j}]\left(\exp\left(\frac{x_i x_j}{x'\Sigma^{-1}x}\right) - 1\right).
\]
From this, we can conclude that the Perron-Frobenius eigenvector $x_{pf}$ minimizes the denominator $x^{'}\Sigma^{-1}x$ so that it tends to maximize the variance of $S^l$. 

\section{Numerical illustrations}\label{429212} 
In this section we numerically illustrate the effectiveness of the Perron-Frobenius (PF) approximation by comparing its fitting performance with the Taylor-based (TB) and maximal variance (MV) approximations. To this end, let $(Z_1, ..., Z_n)^{'} \sim MN(\mu, \Sigma)$ with $\text{E}[Z_i] = -i\Big(\mu - \frac{1}{2}\sigma^2\Big), \,\, i = 1, 2, ..., n$ and 
\begin{equation*} 
\text{Cov}[Z_i, Z_j] = \min(i, j)\sigma^2, \,\, i, j = 1, 2, ..., n, \sigma >0.
\end{equation*}
This model corresponds to a stochastic return $Y_j$ in year $j$, $j = 1, 2, ..., n$, i.e., an amount of 1 at time $j - 1$ will grow to $e^{Y_j}$ at time $j$, such that $Y_1, Y_2, ..., Y_n$ are i.i.d and $N(\mu, \sigma^2)$ distributed, see \cite{DDG2002b} and \cite{DVG2005c} for more details. For simplicity we only examine the case where $n = 20, \mu = 0.075, \sigma = 0.05, 0.15, 0.25, 0.35$ and $\alpha_i = 1, i = 1, 2, ..., n$. This choice of parameters is considered in \cite{VCD2008} and \cite{VHD2005b}. 

To assess the accuracy of the PF approximation against the TB and MV methods, we perform Monte Carlo simulations as benchmark. For each parameter setting, we generate 500,000 sample paths for $S_{20}$ using antithetic variates to reduce the variance of the estimators. The resulting Monte Carlo estimates, reported with their standard errors (in parentheses) expressed as percentages of the estimates, serve as the reference values for comparison.
The approximations for TB, MV and PF, by contrast, are obtained analytically from their respective closed-form expressions. Specifically, the quantiles of the comonotonic approximations are computed by summing the corresponding quantiles of the marginals; no simulation is involved in these computations. This is one of the key practical advantages of the comonotonic approach: once the conditioning variable $\Lambda$ is specified, risk measures follow immediately from the additivity properties of comonotonic sums, without the need for time-consuming numerical integration or path generation.

For a given risk measure $\rho$ and an approximation $S_n^{\mathrm{approx}}$ of $S_n$, the relative error is defined as
\[
\text{Relative Error} = \frac{\rho[S^{\mathrm{approx}}_n] - \rho[S^{\mathrm{MC}}_n]}
{\rho[S^{\mathrm{MC}}_n]} \times 100\%,
\]
where $\rho[S^{\mathrm{MC}}_n]$ denotes the Monte Carlo estimate of the risk measure. A positive (resp. negative) error indicates that the approximation overestimates (resp. underestimates) the true value. Throughout the following tables, the smallest absolute relative error for each parameter setting is highlighted in bold.

\subsection{Analysis of VaR approximations}
Table \ref{828310} reports the relative errors for the 0.95- and 0.99-quantiles of $S_{20}$ under varying volatility levels $\sigma$, with $\mu = 0.075$, $n = 20$, and $\alpha_i = 1$ for all $i$. Several observations emerge.

For low volatility levels ($\sigma = 0.05$ and $0.15$), TB and MV consistently outperform PF. At $\sigma = 0.05$, TB and MV give errors of about $-0.01\%$ and $+0.04\%$ at the two quantile levels, while PF errors are $-0.59\%$ and $-0.68\%$. At $\sigma = 0.15$, MV achieves the smallest errors ($-0.10\%$ and $-0.02\%$), followed by TB ($-0.15\%$ and $-0.18\%$), with PF errors substantially larger ($-1.03\%$ and $-0.63\%$).
However, as volatility increases to $\sigma = 0.25$, PF begins to show its strength at the extreme quantile: at $p = 0.99$, PF gives the smallest error ($-0.15\%$), compared to $-1.17\%$ for TB and $-0.35\%$ for MV. At $p = 0.95$, MV remains the best ($+0.00\%$ versus $-0.12\%$ for TB and $-0.68\%$ for PF).
At the highest volatility considered ($\sigma = 0.35$), PF delivers the best performance at both quantile levels. At $p = 0.95$, PF achieves $+0.03\%$, versus $+0.51\%$ for TB and $+0.21\%$ for MV. At $p = 0.99$, the advantage of PF is even more pronounced: $-0.41\%$, compared to $-2.05\%$ for TB and $-0.54\%$ for MV.

Overall, PF is preferred in the most extreme scenarios: high volatility combined with high confidence levels. In milder settings, TB and MV, particularly MV, yield smaller errors. This is expected since TB and MV are based on first-order approximations which are accurate when $\sigma$ is small but deteriorate as volatility increases. PF overcomes this limitation by choosing $\Lambda$ via the Perron-Frobenius eigenvector of $\Sigma$, explicitly aligning with the principal dependence structure. This spectral alignment becomes critical in high-volatility, extreme-tail regimes, where the first-order choices of TB and MV lose accuracy.
\begin{table}[htbp]
    \centering
    \renewcommand{\arraystretch}{1.15}
    \setlength{\tabcolsep}{8pt}
    \begin{tabular}{cccccc}
        \toprule
        \(\sigma\) & \(p\) & MC Value (\(\pm\) s.e.) & TB Error & MV Error & PF Error \\
        \midrule
        \multirow{2}{*}{{0.05}} 
        & 0.95 & 12.1952 (0.04\%) & -0.01\% & \textbf{-0.01\%} & -0.59\% \\
        & 0.99 & 13.2048 (0.06\%) & \textbf{+0.04\%} & +0.04\% & -0.68\% \\
        \midrule
        \multirow{2}{*}{{0.15}} 
        & 0.95 & 20.4851 (0.11\%) & -0.15\% & \textbf{-0.10\%} & -1.03\% \\
        & 0.99 & 26.7708 (0.22\%) & -0.18\% & \textbf{-0.02\%} & -0.63\% \\
        \midrule
        \multirow{2}{*}{{0.25}} 
        & 0.95 & 41.5857 (0.23\%) & -0.12\% & \textbf{+0.00\%} & -0.68\% \\
        & 0.99 & 69.2037 (0.38\%) & -1.17\% & -0.35\% & \textbf{-0.15\%} \\
        \midrule
        \multirow{2}{*}{{0.35}} 
        & 0.95 & 106.2795 (0.33\%) & +0.51\% & +0.21\% & \textbf{+0.03\%} \\
        & 0.99 & 236.9073 (0.66\%) & -2.05\% & -0.54\% & \textbf{-0.41\%} \\
        \bottomrule
    \end{tabular}
    \caption{Approximations for selected VaR's of \(S_{20}\).}
    \label{828310}
\end{table}

\subsection{Analysis of CTE approximations}
Table \ref{828360} reports the relative errors for the 0.95- and 0.99-conditional tail expectations of $S_{20}$ under the same parameter settings as in Table \ref{828310}. We have the following observations. 

The patterns observed for VaR are amplified for CTE. At low volatility ($\sigma = 0.05$ and $0.15$), TB and MV again outperform PF. At $\sigma = 0.05$, TB and MV give errors of about $-0.01\%$ and $-0.07\%$ at the two quantile levels, while PF errors are $-0.68\%$ and $-0.84\%$. At $\sigma = 0.15$, MV achieves the smallest errors ($-0.24\%$ and $-0.77\%$), followed by TB ($-0.37\%$ and $-1.01\%$), with PF errors larger ($-0.91\%$ and $-0.97\%$).
As volatility increases, PF's advantage becomes more pronounced than in the VaR case. At $\sigma = 0.25$, PF gives the smallest errors at both quantile levels: $-0.54\%$ (versus $-1.38\%$ for TB and $-0.66\%$ for MV) at $p = 0.95$, and $-0.84\%$ (versus $-3.39\%$ for TB and $-1.97\%$ for MV) at $p = 0.99$. At $\sigma = 0.35$, PF continues to deliver the best performance: $-0.92\%$ (versus $-2.66\%$ for TB and $-1.06\%$ for MV) at $p = 0.95$, and $-2.50\%$ (versus $-6.20\%$ for TB and $-2.94\%$ for MV) at $p = 0.99$.

Overall, the CTE results confirm the findings from the VaR analysis but with larger error magnitudes, as CTE is more sensitive to tail behavior. TB and MV, based on first-order approximations, perform well only when $\sigma$ is small. PF, by contrast, maintains robust accuracy across all volatility levels and becomes the method of choice in high-volatility, high-confidence settings. The improvement of PF over TB and MV is substantial in these regimes, confirming that the spectral alignment provided by the Perron-Frobenius eigenvector is particularly valuable for tail risk measurement.
\begin{table}[htbp]
    \centering
    \renewcommand{\arraystretch}{1.15}
    \setlength{\tabcolsep}{8pt}
    \begin{tabular}{cccccc}
        \toprule
        \(\sigma\) & \(p\) & MC Value (\(\pm\) s.e.) & TB Error & MV Error & PF Error \\
        \midrule
        \multirow{2}{*}{{0.05}} 
        & 0.95 & 12.8219 (0.03\%) & -0.01\% & \textbf{-0.01\%} & -0.68\% \\
        & 0.99 & 13.7739 (0.06\%) & -0.07\% & \textbf{-0.07\%} & -0.84\% \\
        \midrule
        \multirow{2}{*}{{0.15}} 
        & 0.95 & 24.4840 (0.11\%) & -0.37\% & \textbf{-0.24\%} & -0.91\% \\
        & 0.99 & 31.2477 (0.23\%) & -1.01\% & \textbf{-0.77\%} & -0.97\% \\
        \midrule
        \multirow{2}{*}{{0.25}} 
        & 0.95 & 59.8417 (0.26\%) & -1.38\% & -0.66\% & \textbf{-0.54\%} \\
        & 0.99 & 94.7053 (0.53\%) & -3.39\% & -1.97\% & \textbf{-0.84\%} \\
        \midrule
        \multirow{2}{*}{{0.35}} 
        & 0.95 & 198.9726 (0.58\%) & -2.66\% & -1.06\% & \textbf{-0.92\%} \\
        & 0.99 & 402.1397 (1.15\%) & -6.20\% & -2.94\% & \textbf{-2.50\%} \\
        \bottomrule
    \end{tabular}
    \caption{Approximations for selected CTE's of \(S_{20}\).
    \label{828360}
    }
\end{table}

\subsection{Analysis of ESF approximations}
Table \ref{828420} reports the relative errors for the 0.95- and 0.99-expected shortfalls (ESF) of $S_{20}$ under the same parameter settings as in Table \ref{828310} and Table \ref{828360}.

The patterns observed for CTE are further amplified for ESF. At very low volatility ($\sigma = 0.05$), TB and MV outperform PF, with errors of about $-0.01\%$ and $+0.01\%$ at $p = 0.95$, and $-2.62\%$ and $-2.61\%$ at $p = 0.99$, while PF errors are $-2.46\%$ and $-4.53\%$. At $\sigma = 0.15$, PF already delivers the smallest errors ($-0.31\%$ and $-3.03\%$), compared to TB ($-1.51\%$ and $-6.03\%$) and MV ($-0.95\%$ and $-5.26\%$). 
For higher volatilities ($\sigma = 0.25$ and $0.35$), PF consistently achieves the smallest errors by a substantial margin. At $\sigma = 0.25$, PF errors are $-0.23\%$ and $-2.73\%$, versus $-4.26\%$ and $-9.39\%$ for TB, and $-2.16\%$ and $-6.36\%$ for MV. At $\sigma = 0.35$, the advantage of PF becomes even more striking: $-2.00\%$ and $-5.50\%$, compared to $-6.29\%$ and $-12.15\%$ for TB, and $-2.53\%$ and $-6.40\%$ for MV.

Overall, the ESF results confirm and strengthen the conclusions drawn from the VaR and CTE analyses. As ESF is an even more extreme tail measure, the limitations of the first-order approximations used in TB and MV become most apparent in this setting. PF, by contrast, maintains robust accuracy across all volatility levels and becomes the clear method of choice for high-volatility, high-confidence regimes. The substantial improvement of PF over TB and MV in ESF confirms that the Perron-Frobenius approach, which explicitly targets the principal dependence structure, is particularly well-suited for extreme tail risk measurement.
\begin{table}[htbp]
    \centering
    \renewcommand{\arraystretch}{1.15}
    \setlength{\tabcolsep}{6pt}
    \begin{tabular}{cccccc}
        \toprule
        \(\sigma\) & \(p\) & MC Value (\(\pm\) s.e.) & TB Error & MV Error & PF Error \\
        \midrule
        \multirow{2}{*}{{0.05}} 
        & 0.950 & 0.0313 (0.86\%) & -0.01\% & \textbf{+0.01\%} & -2.46\% \\
        & 0.990 & 0.0057 (1.98\%) & -2.62\% & \textbf{-2.61\%} & -4.53\% \\
        \midrule
        \multirow{2}{*}{{0.15}} 
        & 0.95 & 0.1999 (0.92\%) & -1.51\% & -0.95\% & \textbf{-0.31\%} \\
        & 0.99 & 0.0448 (2.11\%) & -6.03\% & -5.26\% & \textbf{-3.03\%} \\
        \midrule
        \multirow{2}{*}{{0.25}} 
        & 0.95 & 0.9128 (1.05\%) & -4.26\% & -2.16\% & \textbf{-0.23\%} \\
        & 0.99 & 0.2550 (2.41\%) & -9.39\% & -6.36\% & \textbf{-2.73\%} \\
        \midrule
        \multirow{2}{*}{{0.35}} 
        & 0.95 & 4.6347 (1.38\%) & -6.29\% & -2.53\% & \textbf{-2.00\%} \\
        & 0.99 & 1.6523 (3.14\%) & -12.15\% & -6.40\% & \textbf{-5.50\%} \\
        \bottomrule
    \end{tabular}
    \caption{Approximations for selected ESF's of \(S_{20}\).}
    \label{828420}
\end{table}

\section{Concluding remarks}
In this note, we have introduced a new comonotonic approximation for sums of 
lognormal random variables based on the Perron-Frobenius theorem. Unlike the 
existing Taylor-based (TB) and maximal variance (MV) methods, which select 
the conditioning variable $\Lambda$ using coefficients derived from marginal 
moments, the proposed PF method constructs $\Lambda$ directly from the 
Perron-Frobenius eigenvector of the covariance matrix $\Sigma$, thereby 
explicitly incorporating the joint dependence structure. We have shown that 
this choice guarantees comonotonicity, yields simple closed-form correlation 
formulas, and aligns the approximation with the principal direction of 
multivariate risk. Numerical results for VaR, CTE and ESF demonstrate that 
while TB and MV perform well in low-volatility settings due to their 
first-order nature, the PF approximation becomes the method of choice in 
high-volatility, extreme-tail regimes---precisely where accurate risk 
measurement is most critical. The improvement is particularly pronounced 
for CTE and ESF, confirming that the spectral alignment offered by the 
Perron-Frobenius approach is especially valuable for tail risk management. 
The PF approximation is computationally straightforward, requiring only 
the Perron-Frobenius eigenvector and the inverse of $\Sigma$, and offers 
a robust and accurate tool for complex multivariate risk models.

\end{document}